\documentclass{amsart}
\usepackage{amsmath}
\usepackage{amssymb}
\usepackage{amsfonts}
\usepackage{amscd}
\usepackage{verbatim}
\usepackage{color}

\begin{document}
\topmargin-0.4in
\textheight9.5in
\title{Questions}

\numberwithin{equation}{section}
\newtheorem{thm}[equation]{Theorem}
\newtheorem{cor}[equation]{Corollary}
\newtheorem{lem}[equation]{Lemma}
\newtheorem{prop}[equation]{Proposition}
\newtheorem{qu}[equation]{Question}
\newtheorem{alg}[equation]{Algorithm}
\newtheorem{mproc}[equation]{Procedure}
\newtheorem{defn}[equation]{Definition}
\newtheorem{hyp}[equation]{Standing Assumption}
\newtheorem{rem}[equation]{Remark}
\newtheorem{ex}[equation]{Example}
\newtheorem{exerc}[equation]{Exercise}
\newtheorem{conjecture}[equation]{Conjecture}

\newcommand{\remref}[1]{Remark~\ref{#1}}
\newcommand{\corref}[1]{Corollary~\ref{#1}}
\newcommand{\defref}[1]{Definition~\ref{#1}}
\newcommand{\exercref}[1]{Exercise~\ref{#1}}
\newcommand{\exref}[1]{Example~\ref{#1}}
\newcommand{\lemref}[1]{Lemma~\ref{#1}}
\newcommand{\propref}[1]{Proposition~\ref{#1}}
\newcommand{\thmref}[1]{Theorem~\ref{#1}}
\newcommand{\secref}[1]{Section~\ref{#1}}
\newcommand{\mprocref}[1]{Procedure~\ref{#1}}

\newtheorem{diagl}{Diagrams by Level}
\renewcommand{\thediagl}{}
\newtheorem{diagd}{Diagrams by Depth}
\renewcommand{\thediagd}{}

\newcommand{\ol}{\overline}
\newcommand{\ul}{\underline}
\newcommand{\olsd}[1]{\overline{\overline{#1}}} 
\newcommand{\abar}{\ol{\alpha}}
\newcommand{\eps}{\epsilon}
\newcommand{\lra}{\longrightarrow}
\newcommand{\ra}{\rightarrow}
\newcommand{\rat}{\rightarrowtail}
\newcommand{\thra}{\twoheadrightarrow}
\newcommand{\La}{\Leftarrow}
\newcommand{\Ra}{\Rightarrow}
\newcommand{\Lra}{\Leftrightarrow}
\newcommand{\dhr}{\downharpoonright}
\newcommand{\wdh}{\operatorname{wdh}}
\newcommand{\width}{\operatorname{width}}
\newcommand{\invs}{^{-1}}
\newcommand{\br}{\,\overline{\phantom{n}}\,} 
\newcommand{\brsd}{\,\overline{\overline{\phantom{n}}}\,} 
\newcommand{\dcup}{\overset{\bullet}{\cup}}
\newcommand{\vpl}{\varprojlim}
\newcommand{\defeq}{\overset{\text{\tiny def}}{=}}
\newcommand{\dbf}{\itshape\bfseries} 
\newcommand{\colr}{\color{red}}
\newcommand{\colb}{\color{blue}}
\newcommand{\pht}{\operatorname{hgt}}
\newcommand{\hgt}{\pht}
\newcommand{\ord}{\operatorname{ord}}
\newcommand{\Hgt}[2]{{\mathbb{Q}^{{#1}}_{{#2}}}}
\newcommand{\Hom}{\operatorname{Hom}}
\newcommand{\cHom}{\operatorname{cHom}}
\newcommand{\Ext}{\operatorname{Ext}}
\newcommand{\End}{\operatorname{End}}
\newcommand{\SEnd}{\operatorname{SEnd}}
\newcommand{\Aut}{\operatorname{Aut}}
\newcommand{\rk}{\operatorname{rk}}
\newcommand{\rp}{\operatorname{r_{p}}}
\newcommand{\Coker}{\operatorname{Coker}}
\renewcommand{\Im}{\operatorname{Im}}
\newcommand{\Coim}{\operatorname{Coim}}
\newcommand{\Ker}{\operatorname{Ker}}
\newcommand{\tp}{\operatorname{tp}}
\newcommand{\ts}{\operatorname{Tst}}
\newcommand{\dm}{\dim}
\newcommand{\nriso}{\cong_{\text{\rm nr}}}  
\newcommand{\tpiso}{\cong_{\text{\rm tp}}}  
\newcommand{\quiso}{\cong_{\text{\rm qu}}}  
\newcommand{\IT}{\operatorname{IT}}
\newcommand{\OT}{\operatorname{OT}}
\newcommand{\C}{{\mathbb C}}         
\newcommand{\N}{{\mathbb N}}         
\newcommand{\Nz}{{\mathbb N}_0}      
\newcommand{\Z}{{\mathbb Z}}
\newcommand{\Q}{{\mathbb Q}}
\newcommand{\Reals}{{\mathbb R}}
\newcommand{\Primes}{{\mathbb P}}
\newcommand{\Types}{{\mathbb T}}
\newcommand{\bbT}{{\mathbb T}}
\newcommand{\bbS}{{\mathbb S}}
\newcommand{\bbA}{{\mathbb A}}
\newcommand{\purin}{\vartriangleleft}
\newcommand{\lcm}{\operatorname{lcm}}
\newcommand{\0}{{\bf 0}}

\newcommand{\Mon}{\operatorname{Mon}}
\newcommand{\Stab}{\operatorname{Stab}}
\newcommand{\Orb}{\operatorname{Orb}}
\newcommand{\ReMon}{\operatorname{ReMon}}
\newcommand{\IsoCl}{\operatorname{IsoCl}}
\newcommand{\TIsoCl}{\operatorname{TypIsoCl}}
\newcommand{\TypIsoCl}{\operatorname{TypIsoCl}}
\newcommand{\Tc}{\operatorname{T_{cr}}}
\newcommand{\Tmc}{\operatorname{T_{mcr}}}
\newcommand{\TypAut}{\operatorname{TypAut}}
\newcommand{\TypHom}{\operatorname{TypHom}}
\newcommand{\TypEnd}{\operatorname{TypEnd}}
\newcommand{\SpAut}{\operatorname{SpAut}}
\newcommand{\R}{\operatorname{R}}
\newcommand{\Regg}{\operatorname{Regg}}
\newcommand{\E}{\operatorname{E}}
\newcommand{\Det}{\operatorname{Det}}
\newcommand{\TypH}{\textup{Typ}{\mathbb H}}
\newcommand{\lv}{\operatorname{lv}}     
\newcommand{\dpp}{\operatorname{dp}}    
\newcommand{\dc}{\operatorname{d_{cr}}}
\newcommand{\crgi}{\operatorname{crgi}} 
\newcommand{\rgi}{\operatorname{rgi}}
\newcommand{\xxx}{^{\phantom{-1}}}
\newcommand{\Gen}{\operatorname{Gen}}
\newcommand{\Pinf}{\operatorname{\Primes_\infty}}
\newcommand{\ccy}[1]{\lceil #1 \rceil}  
\newcommand{\rtp}[1]{\|#1\|}            
\newcommand{\Soc}{\operatorname{Soc}}
\newcommand{\Doc}{\operatorname{Doc}}
\newcommand{\Rad}{\operatorname{Rad}}
\newcommand{\NRad}{\operatorname{NRad}}
\newcommand{\hRad}{\operatorname{Rad}}
\newcommand{\tor}{\operatorname{tor}}
\newcommand{\opS}{\operatorname{S}}
\newcommand{\opT}{\operatorname{T}}
\newcommand{\pct}[1]{#1_{\wr p}}        
\newcommand{\rwd}{\operatorname{rwd}}   
\newcommand{\Fx}{\operatorname{Fx}}     
\newcommand{\dust}{}

\newcommand{\CRQ}{{\rm CRQ}} 
\newcommand{\BRCRQ}{{\rm BRCRQ}} 
\newcommand{\RCRQ}{{\rm RCRQ}} 
\newcommand{\fBRCRQ}{{\rm fBRCRQ}} 
\newcommand{\pRCRQ}{{\rm pRCRQ}} 
\newcommand{\olP}{\overline{P}}
\newcommand{\calM}{\mathcal{M}}
\newcommand{\calP}{\mathcal{P}}
\newcommand{\calH}{\mathcal{H}}
\newcommand{\calS}{\mathcal{S}}
\newcommand{\calC}{\mathcal{C}}
\newcommand{\calD}{\mathcal{D}}
\newcommand{\calL}{\mathcal{L}}
\newcommand{\calU}{\mathcal{U}}
\newcommand{\calV}{\mathcal{V}}
\newcommand{\calN}{\mathcal{N}}
\newcommand{\calT}{\mathcal{T}}
\newcommand{\ptn}{\operatorname{ptn}}

\newcommand{\modR}{{\rm Mod}-$R$}
\newcommand{\tffr}{{\rm TFFR\ }}
\newcommand{\LCA}{\mathbf{LCA}}
\newcommand{\QTFFR}{$\Q\mathcal{A}$}
\newcommand{\opM}{\mathbb{M}}
\newcommand{\CC}{\mathcal{C}}
\newcommand{\rh}{\textup{rh}}           
\newcommand{\hc}{\textup{hc}}           
\newcommand{\CDs}{\textup{CD}}          
\newcommand{\FEE}{\textup{FEE}}         
\newcommand{\RFEE}{\textup{RFEE}}       
\newcommand{\CDSG}{\textup{CDSG}}       
\newcommand{\catX}{\mathbf{X}}
\newcommand{\ACDtp}{\textup{ACD}_{tp}}  
\newcommand{\ACDe}{\textup{ACD}_{(e)}}  
\newcommand{\A}{\mathcal{A}}
\newcommand{\Mset}{\mathcal{M}}
\newcommand{\qusum}{\in_\text{\QTFFR}}
\newcommand{\RH}{\textup{RH}}
\newcommand{\cdrepa}{\operatorname{cdrepa}}
\newcommand{\HC}{\operatorname{hc}}
\newcommand{\Reg}{\operatorname{Reg}}
\newcommand{\cdrep}{\operatorname{cdRep}}
\newcommand{\cdRep}{\operatorname{cdRep}}
\newcommand{\hcdrep}{\operatorname{hcdRep}}
\newcommand{\hcdRep}{\operatorname{hcdRep}}
\newcommand{\Rep}{\operatorname{Rep}}
\newcommand{\rep}{\operatorname{Rep}}
\newcommand{\Mat}{\operatorname{Mat}}
\newcommand{\uMat}{\operatorname{uMat}}
\newcommand{\UMat}{\operatorname{uMat}}
\newcommand{\urep}{\operatorname{uRep}}
\newcommand{\uRep}{\operatorname{uRep}}
\newcommand{\UReg}{\operatorname{U-Reg}}

\newcommand\Jac{\operatorname{Jac}}
\newcommand\Nil{\operatorname{\mathcal{N}}}
\newcommand\Id{\operatorname{Idl}}
\newcommand\Mo{\operatorname{Mdl}}
\newcommand\Sin{\operatorname{\Delta}}
\newcommand\subm{\operatorname{SM}}

\newcommand{\krow}[1]{\overset{\rightharpoonup}{#1}}
\newcommand{\kcol}[1]{{\,#1}^\downharpoonright}
\newcommand{\Mring}{\mathcal{M}}             
\newcommand{\dgo}{\operatorname{dgo}}
\newcommand{\transp}{^{\text{tr}}}
\newcommand{\diag}{\operatorname{diag}}
\newcommand{\subdiag}{\operatorname{subdiag}}
\newcommand{\adj}{\operatorname{adj}}
\newcommand{\Ann}{\operatorname{Ann}}
\newcommand{\rteq}{=}
\newcommand{\lteq}{=}
\newcommand{\gcld}{\operatorname{gcld}}
\newcommand{\gcrd}{\operatorname{gcrd}}
\newcommand{\mlt}[1]{\mathbf{#1}}        
\newcommand{\W}{\bigvee\!\!\!\!\bigvee}  
\newcommand{\basis}[1]{\mathcal{#1}}     %
\newcommand{\rmx}[1]{{\EuFrak{m}}_{#1}}  
\newcommand{\DO}[2]{\Pi_{#1,#2}}         
\newcommand{\cred}{\overset{c}{\rightarrow}} 
\newcommand{\rred}{\overset{r}{\rightarrow}} 
\newcommand{\Snf}{\operatorname{Snf}}
\newcommand{\row}{\operatorname{rsp}}
\newcommand{\rsp}{\operatorname{rsp}}

\newcommand{\ann}{\operatorname{Ann}}
\newcommand{\smin}{\subseteq^{\circ}}    
\newcommand{\lgin}{\subseteq^{*}}  
\newcommand{\sdof}{\subseteq^\oplus}     
\newcommand{\pain}{{\rm pi}}             
\newcommand{\Tot}{\operatorname{Tot}}    
\newcommand{\RTD}{RT--decomposition}
\newcommand{\spt}{\operatorname{spt}}
\newcommand{\LR}{\rm LR}
\newcommand{\SR}{\rm SR}
\newcommand{\li}{\rm li}
\newcommand{\lp}{\rm lp}
\newcommand{\efJ}{{\EuFrak{J}}}     
\newcommand{\efI}{{\EuFrak{I}}}     
\newcommand{\efP}{{\EuFrak{P}}}     
\newcommand{\txp}{\rm 2--EP}        
\newcommand{\dtxp}{{\rm D2--EP}}    
\newcommand{\xds}{\overset{\bullet}{\oplus}}  
\newcommand{\bxds}{\overset{\bullet}{\bigoplus}}  
\newcommand{\lstn}{\textrm{lstn}}
\newcommand{\SemiReg}{\operatorname{Semi-Reg}}
\newcommand{\Sing}{\Delta}

\newcommand{\bv}{\big\vert}
\newcommand{\bvv}{\big\vert\!\big\vert}
\newcommand{\Bv}{\Big\vert}
\newcommand{\Bvv}{\Big\vert\!\Big\vert}
\newcommand{\abf}{\bf}
\newcommand{\bst}{\heartsuit}
\newcommand{\btr}{\triangle}
\newcommand{\bsq}{\Box}
\newcommand{\sq}{\Box}
\newcommand{\trf}{\operatorname{trf}}
\newcommand{\Gstar}{\mathcal{G}^*}
\newcommand{\func}{\operatorname}

\newcommand{\gix}{\index}                 
\newcommand{\nix}{\index}                 
\setlength{\parindent}{0mm}

\newcounter{noteno}
\newenvironment{note}{\begin{quotation}
\footnotesize\addtocounter{noteno}{1}
\begin{trivlist}%
\item[\hskip \labelsep {\bf Note \thenoteno}]}
{\end{trivlist}\end{quotation}}

\newcommand{\wn}{\hspace*{0.6mm}0\hspace*{0.6mm}}
\newcommand{\wpi}{\hspace*{0.6mm}p I\hspace*{0.6mm}}
\newcommand{\ptwoi}{\hspace*{0.6mm}p^2{\!}I\hspace*{0.6mm}}
\newcommand{\wprime}{\hspace*{0.6mm}p\hspace*{0.6mm}}
\newcommand{\wptwo}{\hspace*{0.6mm}p^2\hspace*{0.6mm}}


\title[MainDecTG2.tex \quad \today]{The Main Decomposition of
  Finite--Dimensional Protori}

\author {Wayne Lewis}
\address {Department of Mathematics \\ University of Hawaii at Manoa
  \\ 2565 McCarthy Mall, Honolulu, HI 96922, USA}
\email{waynel@math.hawaii.edu} 

\author {Adolf Mader} 
\address {Department of Mathematics \\
  University of Hawaii at Manoa \\ 2565 McCarthy Mall, Honolulu, HI
  96922, USA} 
\email{adolf@math.hawaii.edu} 



\begin{abstract} A \emph{protorus} is a compact connected abelian
  group. We use a result on finite rank torsion-free abelian groups
  and Pontryagin Duality to considerably generalize a well--known
  factorization of a finite dimensional protorus into a product of a
  torus and a torus free complementary factor. We also classify by
  types the solenoids of Hewitt and Ross.
\end{abstract}

\maketitle

\section{Introduction} 
Pontryagin duality restricts to a category equivalence between the
category of discrete abelian groups and the category of compact
abelian groups, \cite[Theorem~7.63]{HM13},
\cite[Proposition~7.5(i)]{HM13}, \cite[(24.3)]{HR63},
\cite[(23.17)]{HR63}. The theory of discrete abelian groups is
well--developed, \cite{F15}, often with a focus on direct
decompositions. Thanks to Pontryagin duality there are equivalent
results valid for compact abelian groups. We will study duals of
completely decomposable abelian groups and in particular the ``Main
Decomposition'' of torsion-free abelian groups of finite rank that
says that every such group is the direct sum of a completely
decomposable summand and a complementary summand that has no
completely decomposable direct summands, \cite[Theorem~2.5]{MS18}. Our
main results are \thmref{main decomposition of compact groups} and
\thmref{main uniqueness}. We also classify solenoids, i.e., one
dimensional connected compact abelian groups, in terms of ``types''
and eliminate redundancies among the $\mathfrak{a}$--adic solenoids of
Hewitt \& Ross, \corref{special a's and type}.

\section{Background} 

All groups considered in this paper are abelian groups, so in the
following ``group'' means ``abelian group''. In the category of
topological groups one has to distinguish between purely algebraic
homomorphisms and algebraic and continuous morphism. By $\Hom(G,H)$ we
denote the set of algebraic morphisms, i.e., the homomorphisms, and by
${\operatorname{cHom}}(G,H)$ we denote the subset of continuous homomorphisms. In the
following $A \cong B$ means that $A,B$ are algebraically isomorphic,
while $G \cong_{t} H$ means that the groups $G,H$ are algebraically
and topologically isomorphic, i.e., isomorphic in the category of
topological groups.  The Pontryagin Dual of a discrete group $A$ is
$A^\vee = \Hom(A,\bbT)$ equipped with the compact--open topology, and
the Pontryagin Dual of a compact group $G$ is the discrete group
$G^\vee = {\operatorname{cHom}}(G,\bbT)$ where $\bbT \overset{\rm def}{=} \Reals/\Z$.

For any topological group $G$, the construct $\mathfrak{L}(G) \overset{\rm
  def}{=} {\operatorname{cHom}}(\Reals,G)$ is its {\dbf Lie Algebra}. It is
well--known that $\mathfrak{L}(G)$ is a real topological vector space via
the stipulation $(r f)(x) \overset{\rm def}{=} f(r x)$ where $f \in
\mathfrak{L}(G)$ and $r,x \in \Reals$, and with the topology of uniform
convergence on compact sets \cite[Definition~5.7,
Proposition~7.36]{HM13}. We define the {\dbf dimension $\dim(G)$ of
  $G$}, by $\dim(G) = \dim_\Reals \mathfrak{L}(G)$.  For a compact group $G$
of positive dimension, $\dim(G) = \rk(\Q \otimes G^\vee)$,
\cite[Theorem~8.22]{HM13}.

\begin{rem}\label{0 dimension} Let $G$ be a compact group. Then the
  following are equivalent. 
\begin{enumerate}
\item $\dim(G) = 0$,
\item $G$ is totally disconnected,
\item $G^\vee$ is a torsion group,
\item $\rk(\Q \otimes G) = 0$,
\end{enumerate} 
\end{rem} 

\begin{proof} (1) $\Lra$ (2): \cite[Proposition~9.46]{HM13}.

(2) $\Lra$ (3): \cite[Corollary~8.5]{HM13}.

(3) $\Lra$ (4): Trivial.
\end{proof}

\section{Short Exact Sequences}

Let $G,H,K$ be topological groups. The sequence $E: H
\overset{f}{\rat} G \overset{g}{\thra} K$ is exact if
\begin{itemize} 
\item $E$ is exact as a sequence in the category of discrete abelian
  groups, i.e., $f$ and $g$ are homomorphisms, $f$ is injective, $g$
  is surjective, and $\Im(f) = \Ker(g)$,
\item $f$ and $g$ are continuous. 
\end{itemize} 

\begin{lem}\label{duals of exact sequences} \hfill 
\begin{enumerate}
\item Let $E: A \overset{f}{\rat} B \overset{g}{\thra} C$ be an exact
  sequence of discrete abelian groups. Then $E^\vee : C^\vee
  \overset{g^\vee}{\rat} B^\vee \overset{f^\vee}{\thra} A^\vee$ is
  exact as a sequence of topological groups where $C^\vee, B^\vee,
  A^\vee$ are compact Hausdorff groups. Furthermore, the maps $g^\vee$
  and $f^\vee$ are open maps onto their images.
\item Let $E: H \overset{f}{\rat} G \overset{g}{\thra} K$ be an exact
  sequence of topological groups, and assume further that $H, G, K$ are
  compact Hausdorff groups. Then $E^\vee : K^\vee
  \overset{g^\vee}{\rat} G^\vee \overset{f^\vee}{\thra} H^\vee$ is
  exact in the category of discrete abelian groups.
\end{enumerate} 
\end{lem} 

\begin{proof} (1) is a consequence of Pontryagin duality and the fact
  that $A^\vee = \Hom(A,\bbT)$ with the compact--open topology and
  $\bbT$ is a divisible abelian group so that $\Ext(C,\bbT)=0$. The
  claim that $g^\vee, f^\vee$ are open maps onto their images follows
  from the Open Mapping Theorem \cite[page 704, EA1.21]{HM13}.

  (2) Here we use \cite[Theorem~2.1]{Mo67}. This theorem requires
  ``proper'' maps, i.e., continuous maps that are open maps onto their
  images. The maps $f$ and $g$ are proper because $H$ and $G$ are
  compact, so that the Open Mapping Theorem applies.
\end{proof}

We will also have to dualize direct sums of discrete groups and
products of compact groups.

\begin{lem}\label{duals of sums and products} \hfill
\begin{enumerate}
\item Let $A = \bigoplus_{i \in I} A_i$ be a direct sum of discrete
  abelian groups. Then $A^\vee = \prod_{i \in I} A_i^\vee$, the
  topology being the product topology. 
\item Let $G = \prod_{i \in I} G_i$ be the topological product of the
  compact groups $G_i$. Then $G^\vee = \bigoplus_{i \in I} G_i^\vee$ in
  the category of discrete abelian groups.
\end{enumerate} 
\end{lem} 

\begin{proof} (1) \cite[Proposition~1.17]{HM13}

  (2) By (1) $\left(\bigoplus_{i \in I} G_i^\vee\right)^\vee \cong
  \prod_{i \in I} G_i^{\vee\vee} \cong \prod_{i \in I} G_i$, hence
  dualizing again $\bigoplus_{i \in I} G_i^\vee \cong \left(\prod_{i
      \in I} G_i\right)^\vee$.
\end{proof}

\begin{lem} \label{dimension sum} If $0 \to G_1 \to G_2 \to G_3 \to
  0$ is an exact sequence of finite dimensional compact abelian
  groups, then $\dim{G_2} = \dim{G_1} + \dim{G_3}$.
\end{lem}

\begin{proof} The exactness of $0 \to G_1 \to G_2 \to G_3 \to 0$
  implies the exactness of $0 \to G_3^\vee \to G_2^\vee \to G_1^\vee
  \to 0$ and this implies the exactness of $0 \to \Q \otimes G_3^\vee
  \to \Q \otimes G_2^\vee \to \Q \otimes G_1^\vee \to 0$ as $\Q$ is
  torsion-free. But this is an exact sequence of $\Q$--vector spaces
  and hence $\dim_\Q (\Q \otimes G_2) = \dim_\Q (\Q \otimes G_3) +
  \dim_\Q (\Q \otimes G_1)$. This establishes the claim because, in
  general $\dim G = \dim_\Q \Q \otimes G^\vee$ by
  \cite[Theorem~8.22]{HM13} for $\dim G \geq 1$ and for $\dim G = 0$
  by \remref{0 dimension}.
\end{proof} 

\section{Totally factored compact Groups} 

A topological group $G$ {\dbf factors} or {\dbf is factorable} if $G
\cong_t \prod_{i} G_i$ with nonzero factors $G_i$ that are
topological groups. If $G$ cannot be factored, then it is {\dbf
  non--factorable}. Connected compact groups of dimension $1$ are
non-factorable by \lemref{dimension sum}. Connected compact groups of
dimension $1$ will be called {\dbf solenoids}.

\begin{prop}\label{solenoids and rational groups}
  The solenoids are exactly the duals of torsion-free groups of
  rank~$1$. In particular, $\Z^\vee = \bbT$, $\bbT^\vee = \Z$. 
\end{prop}

\begin{proof}
  \cite[Corollary~8.5]{HM13}, \cite[Theorem~8.22]{HM13}.
\end{proof}

Torsion-free groups of rank~$1$ are well--understood and can be
classified by means of ``types''. Completely decomposable groups are
direct sums of rank--$1$ groups and, again, are well--understood and
classified up to isomorphism by cardinal invariants,
\cite[Chapter~12.1, Theorem~1.1, Theorem~3.5]{F15}. A connected
compact group is {\dbf totally factored} if it is the topological
product of solenoids. There is a {t}--isomorphism criterion,
\thmref{totally factored}, for totally factored connected compact
groups. These are concepts of the category of compact groups but they
will follow by duality from well--known results in abelian groups.
\medskip

We now review the explicit description of solenoids given in \cite[\S
10]{HR63}, in particular see \cite[(10.12),(10.13)]{HR63}. Let
\[
\mathfrak{a}=(a_0, a_1, a_2, \ldots )\quad \text{where}\quad 2 \leq
a_i \in \Z,
\]
and let $\mathfrak{A}$ be the set of all sequences $\mathfrak{a}$. The
set of {\dbf $\mathfrak{a}$--adic integers} is
\[\textstyle
\Delta_\mathfrak{a} \overset{\text{def}} =\, \prod_{k=0}^\infty
\left\{0,1,\ldots ,a_{k-1} \right\}
\] 
With suitable definition of addition and topology
$\Delta_\mathfrak{a}$ is a compact Hausdorff topological group with
additive identity $\mathbf{0}_\mathfrak{a} = (0,0,0,\ldots)$. Let
$\mathbf{1}_\mathfrak{a} = (1,0,0,\ldots) \in \Delta_\mathfrak{a}$.
(Actually $\Delta_\mathfrak{a}$ can be made into a commutative ring
with $\mathbf{1}_\mathfrak{a}$ its identity.) Define the {\dbf
  $\mathfrak{a}$--adic solenoid}
  \[
  \Sigma_\mathfrak{a} \overset{\rm def}{=} \,\frac{\Delta_\mathfrak{a}
    \times \mathbb{R}}{\Z(\mathbf{1}_\mathfrak{a},1)}
  \] 
  where $\Z(\mathbf{1}_\mathfrak{a},1)$ is the free cyclic subgroup of
  $\Delta_\mathfrak{a} \times \Reals$ additively generated by
  $(\mathbf{1}_\mathfrak{a},1)$. Then $\Sigma_\mathfrak{a}$ is a
  one-dimensional compact connected Hausdorff group,
  \cite[(10.13)]{HR63}. For $\mathfrak{a} \in \mathfrak{A}$ let
  \[
  A_\mathfrak{a} = \left\{\frac{m}{a_0 a_1\cdots a_k}: 0 \le k \in \Z,
    m \in \Z \right\}.
  \]
  Evidently $A_\mathfrak{a}$ is an additive subgroup of $\Q$
  containing $1$, a so-called {\dbf rational group}.

  For $\mathfrak{a} \in \mathbb{A}$, the Pontryagin dual of
  $\Sigma_\mathfrak{a}$ is discrete and isomorphic to
  $A_\mathfrak{a}$, \cite[(25.3)]{HR63}.

\begin{thm}\label{rank one duals} The solenoids $\Sigma_\mathfrak{a}$
  are exactly the duals of the discrete rank--$1$ torsion-free groups
  not isomorphic with $\Z$ while $\bbT^\vee \cong \Z$.
\end{thm}

\begin{proof} By \cite[(25.3)]{HR63} the dual of $\Sigma_\mathfrak{a}$
  where $\mathfrak{a} = (a_0, a_1, \ldots)$ is isomorphic with the
  rational group $A_\mathfrak{a}$. It is left to show that every
  rank--$1$ torsion-free group not isomorphic with $\Z$ is obtained as
  $A_\mathfrak{a}$ for some $\mathfrak{a}$.  Any rank--$1$
  torsion-free group is isomomorphic with a rational group. Therefore
  it suffices to consider a given additive subgroup $A$ of $\Q$
  containing $1$. Let $A$ be such a rational group. It is well--known
  and easy to see (by the partial fraction decomposition of rationals)
  that $A = \langle \frac{1}{p^{h_p}} \mid p \in \Primes\rangle$ where
  $0 \leq h_p \leq \infty$. Here $h_p = p^\infty$ means that all
  fractions $\frac{1}{p^n}$ are included among the generators. Let
  $\Primes_n = \{p \mid 1 \leq h_p < \infty\} = \{p_1, p_2, \ldots\}$
  and let $\Primes_\infty = \{p \mid h_p=\infty\} = \{q_1, q_2,
  \ldots\}$.  Both sets may be empty, finite or infinite. We now set
  \begin{equation}\label{special a's}
  a_0 = p_1^{h_1} q_1,\  a_1 = p_2^{h_2} q_1 q_2,\  a_2 = p_3^{h_3} q_1
  q_2 q_3, \ldots
  \end{equation}
  Then $\mathfrak{a} = (a_0, a_1, a_2, \ldots) \in \mathfrak{A}$ if
  either $\Primes_n$ is infinite or $\Primes_\infty \neq \emptyset$.
  It is evident that
  \[\textstyle
  \left\{\frac{m}{a_0 a_1\cdots a_k}: 0 \le k \in \Z, m \in \Z
  \right\} = \left\langle \frac{1}{p^{h_p}},\frac{1}{q^{\infty}} \mid
    p \in \Primes_n, q \in \Primes_\infty \right\rangle.
  \]
  Note that $\Primes_n$ finite and $\Primes_\infty = \emptyset$ means
  that $A = \langle \frac{1}{p^{h_p}} \mid p \in \Primes\rangle \cong
  \Z$. In this case it is well--known that $\bbT^\vee \cong \Z$.
\end{proof}  

In the proof of \thmref{rank one duals} we saw that a solenoid $\Sigma
\not\cong \bbT$ could be obtained in the form $\Sigma_\mathfrak{h}$ for
the special sequence $\mathfrak{h}$ given in \eqref{special a's}. These
sequences in turn were obtained from $\Sigma^\vee \cong \langle
\frac{1}{p^{h_p}} \mid p \in \Primes\rangle$. The sequence $(h_2, h_3,
h_5, \ldots)$ is the {\dbf height sequence} associated with the
rational group $\Sigma^\vee$. The height sequence of $\bbT^\vee = \Z$
is $\mathfrak{z} = (0,0,0,\ldots)$. Two height sequences are {\dbf
  equivalent} if their entries agree except for finitely many entries
of finite value. The equivalence class of height sequences are {\dbf
 types} and two torsion-free rank--$1$ discrete groups are isomorphic
if and only if they have equal types, \cite[Chapter~12.1,
Theorem~1.1]{F15}. 

Going forward we will assume that any solenoid is given in the form
$\Sigma_\mathfrak{h}$ where $\mathfrak{h}$ arises from the type
$[(h_2,h_3,h_5,\ldots)]$ while $\Sigma_\mathfrak{z} = \bbT$ by
definition.

Let $\Sigma$ be a solenoid. Then the {\dbf type} of $\Sigma$ is
$\tp(\Sigma) \overset{\rm def}{=} [(h_2,h_3,h_5,\ldots)]$ if $\Sigma
\cong \Sigma_\mathfrak{h}$. Thus for a solenoid $\Sigma$ and a
rank--$1$ torsion-free group $A$ we have that $\Sigma^\vee \cong A$ if
and only if $\Sigma \cong A^\vee$ if and only if $\tp(\Sigma) =
\tp(A)$.

The set of types is partially ordered and for types
$\sigma, \tau$ we have $\sigma \leq \tau$ if and only if for rational
groups $A$ of type $\sigma$ and $B$ of type $\tau$ it is true that
$\Hom(A,B) \neq 0$, \cite[Chapter~12, Lemma~1.4]{F15}.

\begin{cor}\label{special a's and type} \hfill 
\begin{enumerate}
\item The set of isomorphism classes of solenoids is in bijective
  correspondence with the partially ordered set $\mathfrak{T}$ of
  types.
\item Given a type $\tau =[(h_2, h_3, h_5, \ldots)] \neq
  [(0,0,0,\ldots)]$ the corresponding solenoid is
  $\Sigma_\mathfrak{h}$ where $\mathfrak{h}$ is given by
  \eqref{special a's}.
\item Two solenoids $\Sigma_1, \Sigma_2$ are isomorphic as topological
  groups if and only if ${\operatorname{cHom}}(\Sigma_1,\Sigma_2) \neq 0 \neq
  {\operatorname{cHom}}(\Sigma_2,\Sigma_1)$.
\end{enumerate} 
\end{cor} 

\begin{proof} (1) and (2) just restate the conclusions reached
  preceding \corref{special a's and type}.  

  (3) ${\operatorname{cHom}}(\Sigma_1,\Sigma_2) \neq 0 \neq {\operatorname{cHom}}(\Sigma_2,\Sigma_1)$
  is equivalent to $\Hom(\Sigma_2^\vee,\Sigma_1^\vee) \neq 0 \neq
  \Hom(\Sigma_1^\vee,\Sigma_2^\vee)$. This implies that
  $\tp(\Sigma_1) = \tp(\Sigma_1^\vee) = \tp(\Sigma_2^\vee) =
  \tp(\Sigma_2)$ and hence $\Sigma_1 \cong_t \Sigma_2$. 
\end{proof} 

The duals of totally factored compact groups are the {\dbf completely
  decomposable groups} of discrete abelian group theory. Completely
decomposable groups can be classified up to isomorphism by cardinal
invariants, \cite[Theorem~3.5]{F15}. We just state a rather
unsatisfactory theorem leaving open the dualizing of the fine points
(type subgroups, types of elements, homogeneous decomposition,
\cite[Chapter~12.3]{F15}) of the theory of totally factored compact
groups. For example, the rank counts the direct summands of a
completely decomposable group hence the dimension counts the factors
in a totally factored compact group and it should be possible to see
this without resorting to duality theory.

\begin{thm}\label{totally factored} A totally factored compact group
  is determined up to isomorphism by types and cardinal invariants.
\end{thm}

On the basis of the known results we consider totally factored compact
groups to be well understood. 

\section{The Main Factorization $G = K \times T$} 

The comprehensive treatise \cite{HM13} contains the following result.

\begin{thm}\label{HM decomposition} Let $G$ be a compact group such
  that $G^\vee$ is torsion-free. Then $G$ is metric and $G \cong_t T
  \times K$ where $K$ is torus--free and $T$ is a characteristic
  maximal torus. In particular, $\Hom(\bbT,K)=0$.
\end{thm} 

\begin{proof} \cite[Theorem~8.45]{HM13} and \cite[Corollary~8.47]{
    HM13}. 
\end{proof}

This result says that $G$ factors with a well--understood factor $T$
and an obscure factor $K$.
 
A compact group is {\dbf clipped} if it has no factor that is a
solenoid. A discrete torsion-free group is {\dbf clipped} if it has
no rank--$1$ summand. Clearly, a compact connected group $G$ is
clipped if and only if $G^\vee$ is clipped. We will establish the
following result.

\begin{thm}\label{main decomposition of compact groups} {\rm (Main
    Decomposition)} Let $G$ be a compact connected Hausdorff group of
  finite dimension. Then $G$ has a decomposition $G = K \times T$ such
  that $T$ is totally factored and $K$ is clipped. In particular,
  ${\operatorname{cHom}}(\bbT,K) = 0$ and ${\operatorname{cHom}}(G,\Q^\vee) = 0$. 
\end{thm}

The result in Hofmann \& Morris is more general in as much as it is
not assumed that $G$ is connected and finite dimensional,
\cite[Corollary~8.5]{HM13}. To illustrate the difference between the
two results we give an example.

\begin{ex} Let $G = \bbT^m \oplus K$ where $K$ is totally factored
but torus--free. Then $G = \bbT^m \oplus K$ is the claimed
decomposition of Hofmann \& Morris where nothing is known about $K$
except that it is torus--free. According to \thmref{main decomposition
  of compact groups} we have the decomposition $G = G \oplus 0$ in
which the obscure factor is trivial.
\end{ex}

\begin{proof} (of \thmref{main decomposition of compact groups}).  Let
  $A = G^\vee$ be the Pontryagin dual of $G$. Then $A$ is torsion-free
  of finite rank. By \cite[Theorem~2.5]{MS18} $A = A_0 \oplus A_1$
  where $A_0$ is completely decomposable and $A_1$ is clipped. Hence
  $G = H_0 \times H_1$ where $H_0 = A_0^\vee$ is totally factored
  and $H_1 = A_1^\vee$ is clipped. The remaining claims follow from
  \lemref{no free or Q summand}.
\end{proof} 

\begin{lem}\label{no free or Q summand} Let $G$ be compact,
  connected of finite dimension and let $A = G^\vee$. \hfill 
\begin{enumerate}
\item $G$ is torus--free if and only if ${\operatorname{cHom}}(\bbT,G)=0$, or,
  equivalently, $\Hom(A,\Z) = 0$.
\item $G$ has a quotient group topologically isomorphic to $\Q^\vee$
  if and only if $A$ has a summand isomorphic to $\Q$.
\item If $G$ is clipped, then ${\operatorname{cHom}}(G,\Q^\vee) = 0$. 
\end{enumerate} 
\end{lem}

\begin{proof} (1) \cite[Corollary~8.47]{HM13}.

  (2) Evident. 

  (3) Note first that $A$ is clipped together with $G$. By way of
  contradiction assume that ${\operatorname{cHom}}(G,\Q^\vee) \ne 0$. Let $0 \ne f \in
  {\operatorname{cHom}}(G,\Q^\vee)$. Then $f$ is proper in the sense of Moskowitz by
  the Open Mapping Theorem as $G$ is compact. Hence $g: G \thra f(G) :
  g=f$ is proper, continuous and surjective. Hence $g^\vee: f(G)^\vee
  \rat G^\vee \cong A$. Because $A$ is torsion-free it follows that
  $f(G)^\vee$ is torsion-free and nonzero. The image $f(G)$ is
  embedded in $\Q^\vee$ and therefore $f(G)^\vee$ is a torsion-free
  image of $\Q$ and hence $f(G)^\vee \cong \Q$ and therefore $A$ has a
  direct summand isomorphic to $\Q$, contradicting that $A$ is
  clipped.
\end{proof}

The factorization of a finite dimensional connected compact group as
``totally factored $\times$ clipped'' has strong uniqueness
properties. We call two compact groups $G, H$ {\dbf nearly isomorphic} if

  $$\forall p\in\Primes  \exists \varphi \in{\operatorname{cHom}}(G,H), \psi\in{\operatorname{cHom}}(H,G), \varphi \psi = n 1_H, \psi \varphi = n 1_G, \gcd(p,n)=1.$$

  The exact same definition is used for discrete torsion-free groups
  of finite rank. Near--isomorphism is weaker than isomorphism but
  preserves important properties such a direct decompositions,
  \cite[Chapter~12.10, \S 10]{F15}. 

\begin{thm}\label{main uniqueness} Let $G$ be a finite dimensional
  connected compact group. Suppose that $G = T \times K = T' \times
  K'$ where $T,T'$ are totally factored and $K,K'$ are clipped. Then
  $T \cong_t T'$ and $K$ and $K'$ are nearly isomorphic.
\end{thm} 

\begin{proof} Dual of part of \cite[Theorem~2.5]{MS18}.
\end{proof}

\end{document}